\documentclass[11pt]{amsart}
\usepackage{amsmath}
\usepackage{amsthm}
\usepackage{amssymb}
\usepackage{graphicx} 
\usepackage{multicol} 
\usepackage{enumerate} 
\usepackage[british]{babel}
\usepackage{tikz-cd} 
\usepackage{mathrsfs} 
\usepackage[hidelinks]{hyperref}
\usepackage[latin1]{inputenc}

\theoremstyle{plain}  
\newtheorem{theorem}{Theorem}[section]
\newtheorem{proposition}[theorem]{Proposition}
\newtheorem{corollary}[theorem]{Corollary}

\newtheorem{conjecture}[theorem]{Conjecture}

\theoremstyle{definition} 
\newtheorem{definition}[theorem]{Definition}
\newtheorem{example}[theorem]{Example}

\theoremstyle{remark} 
\newtheorem{remark}[theorem]{Remark}

\newcommand{\C}{\mathbb{C}}  
\newcommand{\A}{\mathscr{A}} 
\newcommand{\icis}{\textsc{icis} } 
\newcommand{\codimAe}{\textup{codim}_{\A_e}} 
\newcommand{\depth}{\textup{depth}\,} 
\newcommand{\grade}{\textup{grade}\,} 
\newcommand{\Ann}{\textup{Ann}\,} 
\newcommand{\pd}{\textup{pd}} 
\newcommand{\height}{\textup{ht}\,} 
\renewcommand{\O}{\mathscr{O}} 
\newcommand{\F}{\mathscr{F}} 
\newcommand{\m}{\mathfrak{m}} 
\newcommand{\im}{\textup{im}\,} 
\newcommand{\rel}{\textup{rel}} 

\renewcommand{\phi}{\varphi} 
\newcommand{\Derlog}{\textup{Derlog}\,} 

\usepackage{scalefnt}
\let\oldtextsc\textsc
\renewcommand{\textsc}[1]{\oldtextsc{\scalefont{1.1}#1}}

\title{Disentangling mappings defined on ICIS }
\author{Alberto Fern\'andez-Hernández, Juan J. Nu\~no-Ballesteros}
\date{\today}

\address{Departament de Matem\`atiques,
Universitat de Val\`encia, Campus de Burjassot, 46100 Burjassot
SPAIN}
\email{alferher@alumni.uv.es}

\address{Departament de Matem\`atiques,
Universitat de Val\`encia, Campus de Burjassot, 46100 Burjassot
SPAIN. Departamento de Matemática, Universidade Federal da Paraíba 
CEP 58051-900, João Pessoa - PB, BRAZIL}
\email{Juan.Nuno@uv.es}

\thanks{Work of J. J. Nu\~no-Ballesteros partially supported by Grant PID2021-124577NB-I00 funded by MCIN/AEI/ 10.13039/501100011033 and by ``ERDF A way of making Europe"}

\begin{document}
\maketitle
\begin{abstract}
    We study germs of hypersurfaces $(Y,0)\subset (\C^{n+1},0)$ that can be described as the image of $\A$-finite mappings $f:(X,S)\rightarrow (\C^{n+1},0)$ defined on an \icis $(X,S)$ of dimension $n$. We extend the definition of the Jacobian module given by Fernández de Bobadilla, Nuño-Ballesteros and Peñafort-Sanchis when $X=\C^n$, which controls the image Milnor number $\mu_I(X,f)$. We apply these results to prove the case $n=2$ of the generalised Mond conjecture, which states that $\mu_I(X,f)\geq \codimAe (X,f)$, with equality if $(Y,0)$ is weighted homogeneous. 
\end{abstract}

\section{Introduction}
    The main objects that are studied along this paper are hypersurface singularities $(Y,0)\subset (\C^{n+1},0)$. It is well known that if a hypersurface germ has isolated singularity, then one can define two different invariants to encode information of the hypersurface, namely the Milnor and Tjurina numbers, $\mu$ and $\tau$, respectively. The first of them, $\mu$, has a more topological flavour, since it counts the number of spheres in the homotopy type of a Milnor fibre of $Y$. On the other hand, the Tjurina number has a much more rigid nature, and counts the number of parameters of a versal deformation of the hypersurface. An inspection of the algebraic formulas that define them makes clear that $\mu \geq \tau$, with equality in the case that $(Y,0)$ is weighted homogeneous. 
    
    In \cite{vanishingcycles}, D. Mond extended the definition of $\mu$ to a particular kind of hypersurfaces which could be described as the image of a mapping $f:(\C^n,S)\rightarrow (\C^{n+1},0)$ that has isolated instability, or, equivalently, that has finite codimension. In this case, $f$ is the normalisation mapping of the hypersurface $(Y,0)$. He then introduced the concept of \textit{image Milnor number} $\mu_I(f)$ of the mapping $f$, which is well defined if the dimensions $(n,n+1)$ are Mather's nice dimensions, or if $f$ has corank one, as will be commented on in the following lines. This new invariant played the role of the Milnor number for this kind of hypersurfaces that do not have isolated singularity anymore. Moreover, the \textit{codimension} of $f$, denoted by $\codimAe (f)$, is an analytical invariant that measures the minimum number of parameters required to  fully deform $f$ through versal unfoldings, in the same way the Tjurina number does. In terms of this invariant, one says that a mapping is $\A$\textit{-finite} provided it has finite codimension, or that is \textit{stable} whenever it has codimension 0. Although the invariants are written and described in terms of the mapping $f$, they can be seen to depend only on the analytic class of the image $(Y,0)$, and hence $\mu_I(f)$ and $\codimAe(f)$ are the natural invariants that correspond to $\mu$ and $\tau$ in this different context. The question is therefore natural:
    
    \begin{conjecture}[Mond conjecture] Let $f:(\C^n,S)\rightarrow (\C^{n+1},0)$ be an $\A$-finite mapping, where $(n,n+1)$ are Mather's nice dimensions. Then, $\mu_I(f)\geq \codimAe (f)$, with equality if $f$ is weighted homogeneous. 
    \end{conjecture}
    The only known cases are $n=1,2$, and it is still an open problem if $n\geq 3$. This conjecture is therefore the extension of the $\mu\geq \tau $ statement for hypersurfaces that admit an $\A$-finite normalisation mapping with normal space $(\C^n,S)$. Nevertheless, the amount of hypersurfaces satisfying this property is limited. In particular, one would have interest in studying this statement if the normal space of $(Y,0)$ is singular. More precisely, in this article we study the case of hypersurfaces whose normalisation is $\A$-finite and whose normal space is an isolated complete intersection singularity (\textsc{icis}). 
    
    In \cite{Mond-Montaldi}, D. Mond and J. Montaldi settled the deformation theory of mappings $f:(X,S)\rightarrow (\C^p,0)$ with $(X,S)$ being an \textsc{icis} of dimension $n$, and they took the first steps to study an extension of the Mond conjecture in this general setting.
    
    It turns out that the notions of $\codimAe(X,f)$ and $\mu_I(X,f)$ can be defined in this wider context, as Mond and Montaldi showed. The former invariant, namely the codimension of $(X,f)$, equals the minimum number of parameters of a versal unfolding of $(X,f)$, where now unfoldings are allowed to deform both the mapping $f$ and the space $X$. On the other hand, the definition of $\mu_I(X,f)$ is given in the same way as it is done for mappings with smooth source. In order to properly define it, one requires the existence of \textit{stabilisations}, which are deformations $f_t$ of $f$ with the property that $f_t$ is stable for every $t\neq 0$ small enough. These particular deformations do only exist provided $(n,n+1)$ are nice dimensions or provided $f$ has corank one. In these terms, it can be checked that the images $Y_t=\im f_t$, which are called \textit{disentanglements} of $Y$, all have the homotopy type of a wedge of $n$-spheres. The number of such spheres is therefore defined as the image Milnor number $\mu_I(X,f)$. 
    
    This settles the framework to study in this general case whether $\mu_I(X,f)\geq \codimAe (X,f)$, with equality in the weighted homogeneous case. This question is what we refer to as the generalised Mond conjecture. The only case that was known before this article is studied in \cite{juanjo-henrique} by D. Henrique and J.J. Nuño-Ballesteros, and provides a postive answer for the generalised Mond conjecture if $n=1$ and $(X,S)\subset (\C^2,0)$ is a plane curve. 

    Our main contribution in this article is that we prove the generalised Mond conjecture in the case that $n=2$ in its whole generality. Therefore, we show that the $\mu\geq \tau$ statement can be extended to surfaces $(Y,0)$ in $(\C^3,0)$ with $\A$-finite normalisation mapping on an \textsc{icis}. Formally, the main theorem of this paper is the following result:

    \begin{theorem}\label{thrm:main} Let $f:(X,S)\rightarrow (\C^3,0)$ be an $\A$-finite mapping where $(X,S)$ is an \textsc{icis} of dimension $2$ and with image $(Y,0)$. Then, $$\mu_I(X,f)\geq \codimAe (X,f),$$ 
    with equality if $(Y,0)$ is weighted homogeneous. 
    \end{theorem}

    In order to prove it, we adapt the construction of the module $M(g)$ and its relative version for unfoldings $M_\rel (G)$ that were defined in \cite{BobadillaNunoPenafort} by J. Fernández de Bobadilla, J. J. Nuño-Ballesteros and G. Peñafort-Sanchis. We extend to this general framework their main result, which states that the length of $M(g)$ equals $\mu_I(X,f)$ if and only if $M_\rel (G)$ is a Cohen-Macaulay module, and that, if each of these cases hold, the generalised Mond conjecture holds for the mapping $(X,f)$.

\section{Mappings on \textsc{icis}}

Singularities of smooth mappings between manifolds is a classical subject in Singularity Theory. The infinitesimal methods were developed by Thom and Mather in the late sixties and by this reason it is known as Thom-Mather theory. We refer to the book \cite{juanjo} for a modern presentation of the theory, which also includes the extension to holomorphic germs between complex manifolds.

In \cite{Mond-Montaldi} Mond and Montaldi extended the Thom-Mather theory 
of singularities of mappings $f\colon (X,S)\to(\C^p,0)$ defined on an \textsc{icis} $(X,S)$. The crucial point here is that they consider deformations not only of the mapping $f$, but also of the \textsc{icis} $(X,S)$. Here we summarise some of the basic definitions and properties, in order to make this paper more self-contained. For a more detailed account we refer to the original paper by Mond and Montaldi \cite{Mond-Montaldi}.

Along this section we work with holomorphic map germs $f\colon (X,S)\to(\C^p,0)$, where $(X,S)$ is an  \textsc{icis} of dimension $n$. It is usual to denote such a map germs by a pair $(X,f)$, although sometimes we may omit the base set of the germ if it does not provide relevant information or it is clear from the context. The two first definitions declare the type of equivalences and deformations we are dealing with in this theory.

\begin{definition}
Two holomorphic map germs $f,g:(X,S)\rightarrow (\C^p,0)$ are called $\A$\textit{-equivalent} if we have a commutative diagram
$$\begin{tikzcd}
 (X,S) \arrow[r, "f" ]\arrow[d,"\phi"]& (\C^p,0)\arrow[d, " \psi"] \\
 (X,S) \arrow[r, "g" ]& (\C^p,0)
\end{tikzcd}$$
where the columns are biholomorphisms.
\end{definition}

\begin{definition}\label{unfolding}
An \textit{unfolding of the pair }$\left(X,f\right)$ is a map germ $F\colon(\mathcal{X},S')\rightarrow \left(\C^p\times \C^r,0\right)$ together with a flat projection $\pi\colon(\mathcal{X},S')\rightarrow (\C^r,0)$ and an isomorphism $j\colon(X,S)\to\big(\pi^{-1}(0),S'\big)$  such that the following diagram commutes
$$ \begin{tikzcd}[column sep=tiny]
&(X,S) \arrow[dr, "f\times\left\{0\right\}" ] \arrow[dl,  "j" ' ]&  \\
(\pi^{-1}(0),S') \arrow[d,hook] & &\big(\C^p\times\left\{0\right\}, 0\big) \arrow[d,hook]\\
(\mathcal{X},S')\arrow[rr, "F"] \arrow[dr,"\pi" '] && (\C^p\times \C^r,0) \arrow[dl,"\pi_2"]\\
&(\C^r,0)&
\end{tikzcd} ,$$
where $\pi_2:\C^p\times \C^r\rightarrow \C^r$ is the Cartesian projection. \end{definition}

In Definition \ref{unfolding}, $\C^r$ is called the \textit{parameter space of the unfolding}. It is common to denote the unfolding by $(\mathcal{X},\pi,F,j)$. For each parameter $u\in\C^r$ in a neighbourhood of the origin, we have a mapping $f_u\colon X_u\rightarrow \C^p$, where $X_u:= \pi^{-1}(u)$, denoted also by $(X_u,f_u)$.


\begin{definition}
Two unfoldings $(\mathcal{X},\pi,F,j)$ and $(\mathcal{X}',\pi',F',j')$ over $\C^r$ are \textit{isomorphic} if the following diagram commutes:
$$
\begin{tikzcd}[column sep=0.4cm]
 & (\mathcal{X},j(S)) \arrow[dd," \Phi"'] \arrow[rr, "F"]\arrow[dr, "\pi"]& & (\C^p\times \C^r,0)\arrow[dd," \Psi"] \arrow[dl, "\pi_2"']\\
(X,S) \arrow[ur, "j"] \arrow[dr, "j'"'] & & (\C^r,0) & \\
 & (\mathcal{X}',j'(S)) \arrow[rr, "F'"] \arrow[ur, "\pi'"]& & (\C^p\times \C^r,0)\arrow[ul, "\pi_2"']
\end{tikzcd},
$$
where $\Phi$ and $ \Psi$ are biholomorphisms and also $\Psi$ is an unfolding of the identity over $\C^d$.

If $(\mathcal{X},\pi,F,j)$ is an unfolding of $(X,f)$ over $(\C^r,0)$, a germ $\rho: (\C^{s},0) \rightarrow (\C^r,0)$ induces and unfolding $(\mathcal{X}_\rho,\pi_\rho,F_\rho,j_\rho)$ of $(X,f)$ by a \textit{base change} or, in other words, by the fibre product of $F$ and $\text{id}_{\C^p}\times \rho$:
$$\begin{tikzcd}
\mathcal{X}_\rho:= \mathcal{X}\times_{\C^p\times \C^s}\left(\C^p\times \C^s\right)\arrow[r,"F_\rho"]\arrow[d]&\C^p\times \C^s\arrow[d,"\text{id}_{\C^p}\times \rho"]\\
\mathcal{X}\arrow[r,"F"]&\C^p\times \C^r
\end{tikzcd},$$
where we omit the points of the germs for simplicity. 

The unfolding $(\mathcal{X},\pi,F,j)$ is \textit{versal} if every other unfolding, for example $(\mathcal{X}',\pi',F',j')$, is isomorphic to an unfolding induced from the former by a base change, $(\mathcal{X}_\rho, \pi_\rho, F_\rho, j_\rho)$. A versal unfolding is called \textit{miniversal} if it has a parameter space with minimal dimension.
\end{definition}

\begin{definition} 
A germ $(X,f)$ is \emph{stable} if any unfolding is \emph{trivial}, that is, isomorphic to the constant unfolding  $(X\times\C^r,\pi_2,f\times\textnormal{id}_{\C^r},i)$.

We say that $(X,f)$ has \emph{isolated instability} if there exists a representative $f\colon X\to\C^p$ such that the restriction $f\colon X\setminus f^{-1}(0)\to \C^p\setminus\{0\}$ has only stable singularities.

A \emph{stabilisation} of $(X,f)$ is a 1-parameter unfolding $(\mathcal{X},\pi,F,j)$ with the property that for any small enough $s\in\C\setminus\{0\}$, $(X_s,f_s)$ has only stable singularities. Such a mapping $(X_s,f_s)$, with $s\ne0$, is called a \emph{stable perturbation} of $(X,f)$. 
\end{definition}

A crucial fact is that any germ $(X,f)$ with isolated instability admits a stabilisation, provided that $(n,p)$ are nice dimensions in the sense of Mather or $f$ has only kernel rank one singularities (that is, $f$ admits an extension whose differential has kernel rank $\le 1$ everywhere). A proof in the case $X=\C^n$ can be found in \cite{juanjo} and the extension to the case of mappings on \textsc{icis} appears in \cite{roberto}.

Next, we recall the notion of $\A_e$-codimension of a germ $(X,f)$. In order to do this, we introduce the following notation:
\begin{itemize}
\item $\O_p$ is the local ring of holomorphic functions $(\C^p,0)\to\C$,
\item $\O_{X,S}$ is the (semi-)local ring of holomorphic functions $(X,S)\to\C$,
\item $f^*:\O_p\to\O_{X,S}$ is the induced ring morphism $f^*(h)=h\circ f$,
\item $\theta_{\C^p,0}$ is the $\O_p$-module of germs of vector fields on $(\C^p,0)$,

\item $\theta_{X,S}$ is the $\O_{X,S}$-module of germs of vector fields on $(X,S)$,

\item $\theta(f)$ is the module of vector fields along $f$,

\item $\omega f:\theta_{\C^p,0}\rightarrow \theta(f)$ is the mapping $\omega(\eta)=\eta\circ f$, 

\item $tf:\theta_{X,S}\rightarrow \theta(f)$ is mapping $tf(\xi)=d\tilde f\circ\xi$, for some analytic extension $\tilde f$ of $f$.
\end{itemize}

\begin{definition}
The \emph{$\A_e$-codimension} of $(X,f)$ is defined as 
\[
\codimAe(X,f)=\dim_\C\frac{\theta(f)}{{tf(\theta_{X,S})+\omega f(\theta_p)}}+\sum_{x\in S} \tau(X,x),
\]
where $\tau(X,x)$ is the Tjurina number of $(X,x)$. When $\codimAe(X,f)<\infty$, the germ $(X,f)$ is called {$\A$-finite}.
\end{definition}

The versality theorem holds also for mappings on \textsc{icis}, as the reader can find in \cite{Mond-Montaldi}: $(X,f)$ is $\A$-finite if and only if it admits a versal unfolding and in case it is $\A$-finite, then $\codimAe(X,f)$ is equal to the  number of parameters in a miniversal unfolding. As a consequence, $(X,f)$ is stable if and only if $(X,S)$ is smooth and $f$ is stable in the usual sense (see \cite{juanjo}).

Another important issue with $\A$-finiteness is the extension of the Mather-Gaffney geometric criterion for mappings on \textsc{icis}: $(X,f)$ is $\A$-finite if and only if it has isolated instability (see \cite{roberto}).

Finally, we will recall the definition of image Milnor number in the case $p=n+1$. The original definition when $X=\C^n$ is due to Mond (see \cite{vanishingcycles, juanjo}), but it can be adapted quite easily to mappings on \textsc{icis} (see \cite{roberto}). In the case $p\le n$, the analogous invariant is called the discriminant Milnor number, considered for the first time by Damon and Mond in \cite{Damon-Mond} in the case $X=\C^n$ and extended to mappings on \textsc{icis} by Mond and Montaldi in \cite{Mond-Montaldi}. The definition is clearly inspired in the classical Milnor number and is motivated by the following theorem.

We denote by $B_\epsilon$ the closed ball in $\C^{n+1}$ of radius $\epsilon>0$ centered at the origin. We assume $f\colon (X,S)\to(\C^{n+1},0)$ is $\A$-finite and that either $(n,n+1)$ are nice dimensions of Mather or $f$ has only corank one singularities. We take a stabilisation of $(X,f)$ with stable perturbation $(X_s,f_s)$.

\begin{theorem}\label{disent}\cite{vanishingcycles, roberto} For all $\epsilon,\eta$, with $0<\eta\ll\epsilon\ll 1$ and for all $s\in\C$, with $0<|s|<\eta$, $f_s(X_s)\cap B_\epsilon$ has the homotopy type of a bouquet of spheres of dimension $n$. Moreover, the number of such spheres, denoted by $\mu_I(X,f)$, is independent of the choice of the stabilisation, the parameter $s$  and the numbers $\epsilon,\eta$.
\end{theorem}

\begin{definition} With the notation of Theorem \ref{disent}, $f_s(X_s)\cap B_\epsilon$ is called the \emph{disentanglement} and $\mu_I(X,f)$ is called the \emph{image Milnor number} of $(X,f)$.
\end{definition}

The proof of Theorem \ref{disent} is based on arguments by Lê \cite{Le} and by Siersma \cite{Siersma}. In fact, the original formulation in \cite{Siersma} gives a recipe of how to compute $\mu_I(X,f)$ in a more algebraic way:

\begin{theorem}\cite{Siersma}\label{Siersma} With the notation of Theorem \ref{disent}, let $G\colon(\C^{n+1}\times\C,0)\to(\C,0)$ be a reduced equation of the image of the stabilistation. Then,
\[
\mu_I(X,f)=\sum_{y\in B_\epsilon \setminus f_s(X_s)} \mu (g_s; y),  
\]
where $g_s(y)=G(y,s)$ and $ \mu (g_s; y)$ is the Milnor number of $g_s$ at $y$.
\end{theorem}

\begin{remark} In order to compute the image Milnor number $\mu_I(X,f)$, sometimes is more convenient to consider a stable unfolding (i.e., an unfolding which is stable as a germ) instead of a stabilisation. In such a case, the bifurcation set $\mathcal B$ is the set of parameters $u\in\C^r$ in a neighbourhood of the origin, such that $f_u\colon X_u\to \C^{n+1}$ has only stable singularities. When $(n,n+1)$ are nice dimensions or when $f$ has only corank one singularities, $\mathcal B$ is a proper closed analytic subset germ in a neighbourhood of $0$ in $\C^r$ (see \cite{juanjo} or \cite{roberto}). A stabilisation can be constructed easily by taking a line $L\subset\C^r$ with the property that $L\cap \mathcal B=\{0\}$. If $u\notin\mathcal B$, then $f_u(X_u)\cap B_\epsilon$ has the homotopy type of a bouquet of $n$-spheres and the number of such spheres is $\mu_I(X,f)$. Analogously, if $G\colon(\C^{n+1}\times\C^r,0)\to(\C,0)$ is a reduced equation of the image of the unfolding, then
\[
\mu_I(X,f)=\sum_{y\in B_\epsilon \setminus f_u(X_u)} \mu (g_u; y),  
\]
where $g_u(y)=G(y,u)$ and $ \mu (g_u; y)$ is the Milnor number of $g_u$ at $y$.
\end{remark}

\section{The generalised version of the Jacobian module $M(g)$}\label{section3}

Let $f:(X,S)\rightarrow (\C^{n+1},0)$ be an $\A$-finite mapping defined on an \icis $(X,S)\subset (\C^{n+k},S)$ of dimension $n$. As $f$ is finite, its image $(Y,0)$ is a hypersurface of $(\C^{n+1},0)$, and hence it can be described by a reduced equation $g\in \O_{n+1}$. Let $h:(\C^{n+k},S)\rightarrow (\C^k,0)$ be a mapping so that $(X,S)=h^{-1}(0)$. Consider an analytic extension $\tilde{f}:(\C^{n+k},S)\rightarrow (\C^{n+1},0)$ of $f$, and write $\hat{f}=(\Tilde{f},h):(\C^{n+k},S)\rightarrow (\C^{n+1+k},0)$. It is therefore clear that the restriction of $\hat{f}$ to $(X,S)$ is precisely the mapping $(f,0)$. Hence, $\hat{f}$ is a finite mapping, since $\hat{f}^{-1}(0)=f^{-1}(0)=S$. Moreover, the diagram
\begin{center}
    \begin{tikzcd}
	& {(\C^k,0)} \\
	{(\C^{n+k},S)} & {(\C^{n+1+k},0)} \\
	{(X,S)} & {(\C^{n+1},0)}
	\arrow["f", from=3-1, to=3-2]
	\arrow["i", hook, from=3-1, to=2-1]
	\arrow["\hat{f}", from=2-1, to=2-2]
	\arrow["j"', hook, from=3-2, to=2-2]
	\arrow["h", two heads, from=2-1, to=1-2]
	\arrow[two heads, from=2-2, to=1-2]
\end{tikzcd}
\end{center}
commutes, where $i$ is the inclusion and $j$ is the natural immersion. In particular, $\hat{f}$ is an unfolding of $f$ deforming both the mapping and the domain. After taking representatives, the induced deformations of $f$ are the mappings $\hat{f}_t:X_t\subset \C^{n+k}\rightarrow \C^{n+1}$ defined as $\hat{f}_t(x)=\tilde{f}(x,t)$, where $X_t=h^{-1}(t)$ for $t\in \C^k$ small enough. 

The key idea is that $\hat{f}$ is the simplest unfolding of $f$ with smooth source. Hence, the definition of the module for the mapping $f$ will be performed through a specialisation of the one from $\hat{f}$. 

Since $\hat{f}$ is a finite mapping, its image $(\hat{Y},0)$ is a hypersurface of $(\C^{n+1+k},0)$. Furthermore, the restriction of $\hat{f}$ to its image $(\C^{n+k},S)\rightarrow (\hat{Y},0)$ is the normalisation mapping of $(\hat{Y},0)$, and hence it induces a monomorphism of rings $\O_{\hat{Y},0}\hookrightarrow \O_{n+k}$ that lets us consider $\O_{\hat{Y},0}$ to be a subring of $\O_{n+k}$. In this case, the diagram
\begin{equation*}
\begin{tikzcd}
\O_{n+1+k} \arrow[r, "\hat{f}^*"] \arrow[dr,"\pi",two heads]
&  \O_{n+k}  \\
& \O_{\hat{Y},0}\arrow[u,hook]
\end{tikzcd}
\end{equation*}
commutes, where $\pi$ is the epimorphism associated to the natural inclusion of $(\hat{Y},0)$ in $(\C^{n+1+k},0)$. We consider both $\O_{\hat{Y},0}$ and $\O_{n+k}$ to be $\O_{n+1+k}$-modules via the corresponding morphisms $\pi$ and $\hat{f}^*$, respectively. 

Let us consider $\hat{g} \in \O_{n+1+k}$ to be a reduced equation of $(\hat{Y},0)$ in such a way that $\hat{g} \circ j =g$. The following result from R. Piene \cite{Piene} relates the conductor ideal of $\hat{f}$, given by 
\begin{equation*}
    C(\hat{f})= \{ h \in \O_{\hat{Y},0} : h\cdot \O_{n+k} \subset \O_{\hat{Y},0}\},
\end{equation*}
with the partial derivatives of $\hat{g}$ and with the minors of the Jacobian matrix $d\hat{f}$.
\begin{theorem}[\cite{Piene}]
    There exists a unique $\lambda \in \O_{n+k}$ such that, for every $l\in \{1, \ldots, n+1+k\}$, 
    \begin{equation*}
        \partial_l \hat{g}\circ \hat{f}=(-1)^l\cdot \lambda \cdot\det (d\hat{f}_1, \ldots, d\hat{f}_{l-1}, d\hat{f}_{l+1}, \ldots, d\hat{f}_{n+1+k}),
    \end{equation*}
    where $\partial_l \hat{g}$ denotes the partial derivative of $\hat{g}$ with respect to the $l$-th variable. Furthermore, the ideal $C(\hat{f})$ is principal, and generated by the element $\lambda$. 
\end{theorem}

Let us denote by $J(\hat{g})$ to the Jacobian ideal of $\hat{g}$, which is generated by the partial derivatives $\partial_l \hat{g}$. This result therefore shows that $J(\hat{g})\cdot \O_{n+k}\subset C(\hat{f})$, where $J(\hat{g})\cdot \O_{n+k} = \hat{f}^*(J(\hat{g}))$. 

On the other hand, since $\hat{f}$ is a finite mapping with degree 1 onto its image, an application of Theorem 3.4 of \cite{Mond-Pellikaan} of D. Mond and R. Pellikaan shows that $\F_1(\hat{f}) \cdot \O_{n+k} = C(\hat{f})$, where $\F_1(\hat{f})$ denotes the first Fitting ideal of $\O_{n+k}$ as an $\O_{n+1+k}$-module via $\hat{f}^*$. 

Let us denote by $(y,z)$ to the coordinates in $(\C^{n+1+k},0)$, with $y\in \C^{n+1}$ and $z\in \C^k$, and consider the Jacobian ideal $J_y(\hat{g} ) = \langle \partial \hat{g}/\partial y_1, \ldots, \partial \hat{g}/\partial y_{n+1}\rangle$ generated by the derivatives with respect to the variables $y=(y_1, \ldots, y_{n+1})$. 
\begin{definition}\label{def:module} The restriction of $\hat{f}^*$ to $\F_1(\hat{f})$ induces an epimorphism of $\O_{n+1+k}$-modules 
\begin{equation*}
    \dfrac{\F_1(\hat{f})}{J_y(\hat{g})}\rightarrow \dfrac{C(\hat{f})}{J(\hat{g})\cdot \O_{n+k}}.
\end{equation*}
We define $N(\hat{g})$ to be the $\O_{n+1+k}$-module given by the kernel of this morphism, and define 
\begin{equation*}
    M(g)=N(\hat{g})\otimes\dfrac{\O_{n+1+k}}{\m_k\cdot \O_{n+1+k}},
\end{equation*}
where $\m_k=(z_1, \ldots, z_k)$ is generated by the parameters of the unfolding $\hat{f}$ of $f$. Note that $M(g)$ has a natural $\O_{n+1}$-module structure inherited from the tensor product. 
\end{definition}
Hence, $M(g)$ is defined by taking into account that $\hat{f}$ is an unfolding of $f$, and following the spirit of \cite{BobadillaNunoPenafort} in the case of mappings with smooth source. Furthermore, it is relevant to notice that this module $M(g)$ coincides with the given in the smooth case just by taking $\hat{f}=f$ and $k=0$.

\begin{remark} Since $\hat{f}$ is defined in a smooth source, it lies naturally in the context of \cite{BobadillaNunoPenafort}. However, the module $N(\hat{g})$ that we have defined is not exactly the same as the module that is defined in \cite{BobadillaNunoPenafort}. Indeed, the source of the morphism that defines $N(\hat{g})$ has the Jacobian ideal $J_y(\hat{g})$ where only partial derivatives with respect to the parameters $y_1, \ldots, y_{n+1}$ are taken into consideration, while the module in the target does have all the partial derivatives. Although this may seem to be somewhat whimsical, this will be shown to be exactly what is needed for the module to specialise properly. This is due to the fact that, in general, $J(\hat{g})\cdot \O_{n+k}\neq J_y(\hat{g})\cdot \O_{n+k}$, in contrast with what happens in the smooth case. This important fact is what infuences this definition for $M(g)$, and what makes that another possible definition may not specialise properly. 
\end{remark}
\begin{remark} The module $N(\hat{g})$ is determined by the short exact sequence
\begin{equation*}
    0\rightarrow N(\hat{g}) \rightarrow \dfrac{\F_1(\hat{f})}{J_y(\hat{g})}\rightarrow \dfrac{C(\hat{f})}{J(\hat{g})\cdot \O_{n+k}}\rightarrow 0.
\end{equation*}
\end{remark}
\begin{proposition} The following formula holds:
\begin{equation*}
    N(\hat{g})=\dfrac{(\hat{f}^*)^{-1} (J(\hat{g})\cdot \O_{n+k})}{J_y(\hat{g})}. 
\end{equation*}
\end{proposition}
\begin{proof} By definition, it is clear that 
\begin{equation*}
    N(\hat{g})=\dfrac{(\hat{f}^*)^{-1} (J(\hat{g})\cdot \O_{n+k})\cap \F_1(\hat{f})}{J_y(\hat{g})}.
\end{equation*}
Since $J(\hat{g})\cdot \O_{n+k}\subset C(\hat{f})$, then $(\hat{f}^*)^{-1} (J(\hat{g})\cdot \O_{n+k})\subset  \F_1(\hat{f})$, and the formula follows.
\end{proof}

\section{A relative version for the generalised Jacobian module}

In this section, a relative version of the module for unfoldings of $\hat{f}$ is defined. For the sake of simplicity, we consider unfoldings $F:(\C^{n+k+r},S\times 0)\rightarrow (\C^{n+1+k+r},0)$ of the mapping $\hat{f}$ instead of general unfoldings of $f$ with possibly non-smooth source. Let $G\in \O_{n+1+k+r}$ be an equation for the image of $F$ such that $G(y,z,0)=\hat{g}(y,z)$, where $(y,z,u)$ are the coordinates of $(\C^{n+1+k+r},0)$, with $y\in \C^{n+1}, z\in \C^k$ and $u\in \C^r$. We then have that the diagram
\[\begin{tikzcd}
	{(\C^{n+k+r},S\times 0)} & {(\C^{n+1+k+r},0)} \\
	{(\C^{n+k},S)} & {(\C^{n+1+k},0)} & {(\C,0)} \\
	{(X,S)} & {(\C^{n+1},0)}
	\arrow["{\hat{f}}", from=2-1, to=2-2]
	\arrow["f", from=3-1, to=3-2]
	\arrow["i", hook, from=3-1, to=2-1]
	\arrow["j"', hook, from=3-2, to=2-2]
	\arrow["{\hat{i}}", hook, from=2-1, to=1-1]
	\arrow["{\hat{j}}"', hook, from=2-2, to=1-2]
	\arrow["F", from=1-1, to=1-2]
	\arrow["g", from=3-2, to=2-3]
	\arrow["{\hat{g}}", from=2-2, to=2-3]
	\arrow["G", from=1-2, to=2-3]
\end{tikzcd}\]
commutes, where $\hat{i}, \hat{j}$ are the natural immersions. Let us denote by 
\begin{equation*}
    J_y(G ) = \Big\langle \dfrac{\partial G}{\partial y_1}, \ldots, \dfrac{\partial G}{\partial y_{n+1}}\Big\rangle, \,\,
    J_z(G ) = \Big\langle \dfrac{\partial G}{\partial z_1}, \ldots, \dfrac{\partial G}{\partial z_k}\Big\rangle, 
\end{equation*}
and $J_{y,z}(G)=J_y(G)+J_z(G)$. We therefore define $M_\rel (G)$ as the kernel of the module epimorphism 
\begin{equation*}
    \dfrac{\F_1(F)}{J_y (G)}\rightarrow \dfrac{C(F)}{J_{y,z}(G)\cdot \O_{n+k+r}}.
\end{equation*}
Hence, $M_\rel (G)$ fits into the short exact sequence 
\begin{equation*}
    0\rightarrow M_\rel (G) \rightarrow \dfrac{\F_1(F)}{J_y (G)}\rightarrow \dfrac{C(F)}{J_{y,z}(G)\cdot \O_{n+k+r}}\rightarrow 0.
\end{equation*}
Furthermore, it is straightforward to verify that 
\begin{proposition}\label{prop:formulaMrelG} The following formula holds:
\begin{equation*}
    M_\rel (G)=\dfrac{(F^*)^{-1} (J_{y,z}(G)\cdot \O_{n+k+r})}{J_y(G)}.
\end{equation*}
\end{proposition}
In contrast with the case of $M(g)$, we have that $J_{y,z}(G)\cdot \O_{n+k+r}=J(G)\cdot \O_{n+k+r}$ in virtue of lemma 4.5 of \cite{BobadillaNunoPenafort}. Hence, they can be interchanged indistinctively in this formula. 

Notice, in addition, that the relative module in \cite{BobadillaNunoPenafort} is denoted by $M_y(G)$ with the intention to highlight that the module is relative with respect to the parameters $y=(y_1, \ldots, y_{n+1})$. Since this interpretation in the case of mappings with non-smooth source is less clear, due to the fact that one has both parameters $y,z$, the authors have decided to consider the notation $M_\rel (G)$. 

In what follows, we show that the given definition of the $\O_{n+1+k+r}$-module $M_\rel (G)$ properly specialises to $M(g)$ seen as an $\O_{n+1}$-module. Before giving a proof of this result, let us comment on the specialisation method that should be considered. 

\begin{remark}\label{remark:specialisation}
It is relevant to notice that the specialisation process that is performed in \cite{BobadillaNunoPenafort} for mappings with smooth source restricts the relative module $M_y(G)=M_\rel (G)$ to be an $\O_r$-module, and then $M_\rel(G)$ is tensored with $\O_r/\m_r$. Hence, one naturally obtains a $\C$-vector space that is isomorphic to $M(g)$ (see Theorem 4.6 of \cite{BobadillaNunoPenafort}). However, this specialisation process ignores the fact that $M(g)$ is an $\O_{n+1}$-module. In order to take this into account, one should perform the specialisation process in a slightly different way. Since $M_\rel (G)$ is an $\O_{n+1+r}$-module, the tensor product 
\begin{equation*}
    M_\rel(G)\otimes \dfrac{\O_{n+1+r}}{\m_r\cdot \O_{n+1+r}}
\end{equation*}
provides naturally an $\O_{n+1}$-module, due to the fact that $\O_{n+1+r}/(\m_r\cdot \O_{n+1+r})\cong \O_{n+1}$, where $\m_r=(u_1, \ldots , u_r)$ is the maximal ideal of $\O_r$ generated by the parameters of the unfolding $F$. In fact, this specialisation process both captures the spirit of forcing the parameters of the unfolding to be equal to 0 and keeps the natural $\O_{n+1}$-module structure of $M(g)$, as it can be easily verified with minor modifications in the proofs of section 4 of \cite{BobadillaNunoPenafort}. Hence, it can be easily checked that the module $M_\rel (G)$ of \cite{BobadillaNunoPenafort} in the smooth case satisfies that 
\begin{equation*}
    M_\rel(G)\otimes \dfrac{\O_{n+1+r}}{\m_r\cdot \O_{n+1+r}} \cong M(g),
\end{equation*}
where $\cong $ denotes isomorphism of $\O_{n+1}$-modules. Although this is not relevant in the smooth case, we have already seen in the definition \ref{def:module} that specialisations should be taken through this method in order to preserve the module structure.
\end{remark}

Taking this into account, we are now able to check that the module $M_\rel (G)$ in this setting properly specialises to $M(g)$:
\begin{theorem}
    If $n\geq 2$, then 
    \begin{equation*}
        M_\rel (G) \otimes \dfrac{\O_{n+1+k+r}}{\m_{k+r}\cdot \O_{n+1+k+r}}\cong M(g)
    \end{equation*}
    as $\O_{n+1}$-modules.
\end{theorem}
\begin{proof} The same proof of the smooth case (Theorem 4.6 of \cite{BobadillaNunoPenafort} and the previous remark) shows that 
\begin{equation*}
    M_\rel (G)\otimes \dfrac{\O_{n+1+k+r}}{\m_r\cdot \O_{n+1+k+r}}\cong N(\hat{g}),
\end{equation*}
since this first specialisation restricts the parameters that were added in the unfolding $F$ of $\hat{f}$, both with smooth domain. Indeed, the only difference is that the definition of $N(\hat{g})$ has a term $J_y(\hat{g})$, but it holds that $\hat{j}(J_y(G))=J_y(\hat{g})$ and, as in the smooth case, $\hat{j}(J_{y,z}(G))=J(\hat{g})$. Hence, 
\begin{align*}
   M_\rel (G) \otimes \dfrac{\O_{n+1+k+r}}{\m_{k+r}\cdot \O_{n+1+k+r}} &= \left( M_\rel (G) \otimes \dfrac{\O_{n+1+k+r}}{\m_{r}\cdot \O_{n+1+k+r}}\right)\otimes \dfrac{\O_{n+1+k}}{\m_{k}\cdot \O_{n+1+k}} \\
   &= N(\hat{g})\otimes \dfrac{\O_{n+1+k}}{\m_{k}\cdot \O_{n+1+k}}  =M(g),
\end{align*}
where the last equality holds by definition of $M(g)$. 
\end{proof}

\begin{remark}\label{remark:specialisation2} As we have commented on before in Remark \ref{remark:specialisation}, the specialisation performed is different than the one appearing in \cite{BobadillaNunoPenafort} for mappings with smooth source. Although this different approach lets us transfer the module structure, it is important for applications to analyse the analogue process. It turns out that, when we restrict $M_\rel (G)$ to be an $\O_{k+r}$-module via the natural inclusion $\O_{k+r}\rightarrow \O_{n+1+k+r}$ induced by the projection $\C^{n+1}\times \C^{k+r}\rightarrow \C^{k+r}$, then it is straightforward to verify that 
\begin{equation*}
    M_\rel(G)\otimes \dfrac{\O_{k+r}}{\m_{k+r}}\cong M(g)
\end{equation*}
as $\C$-vector spaces, just by noticing that 
\begin{align*}
    M_\rel(G)\otimes \dfrac{\O_{k+r}}{\m_{k+r}}&\cong \dfrac{M_\rel (G)}{\m_{k+r}\cdot M_\rel (G)} = \dfrac{M_\rel (G)}{(\m_{k+r}\cdot \O_{n+1+k+r})\cdot M_\rel (G)} \cong \\
    &\cong M_\rel (G) \otimes \dfrac{\O_{n+1+k+r}}{\m_{k+r}\cdot \O_{n+1+k+r}} \cong M(g).
\end{align*}
In a nutshell, the specialisation process can be performed either restricting first $M_\rel (G)$ to be an $\O_{k+r}$-module, or rather keeping its whole $\O_{n+1+k+r}$-module structure. The obtained modules are, respectively, 
$$M_\rel (G)\otimes (\O_{n+1+k+r}/\m_{k+r}\cdot \O_{n+1+k+r}) \text{ and } M_\rel (G)\otimes (\O_{k+r}/\m_{k+r}).$$ 
In both scenarios one recovers $M(g)$ with some structure: in the former case, the result is an $\O_{n+1}$-module isomorphic to $M(g)$, and, in the latter, a $\C$-module isomorphic to $M(g)$. Hence, the specialisation procces that one should perform depends on whether one needs to keep the module structure or not. For most of the cases, the only information that needs to be keeped is the complex dimension as a vector space, and hence both methods are valid. 
\end{remark}

Lastly, an important result regarding the form of the module $M_\rel (G)$ when $F$ is stable is the following:
\begin{proposition} Let $F$ be a stable unfolding of $\hat{f}$ and $G$ an equation such that $G\in J(G)$. Then, 
\begin{equation*}
    M_\rel (G)= \dfrac{J(G)}{J_y(G)}.
\end{equation*}
\end{proposition}
\begin{proof}
    Since $F$ is stable and $G\in J(G)$, the smooth version of the module $M(G)$ from \cite{BobadillaNunoPenafort} satisfies that $M(G)=0$. In this case, the formula of proposition 5.1 of this article yields that
    \begin{equation*}
        M(G)=\dfrac{(F^*)^{-1}(J(G)\cdot \O_{n+k+r})}{J(G)},
    \end{equation*}
    and hence $(F^*)^{-1}(J(G)\cdot \O_{n+k+r})=J(G)$. Furthermore, in the smooth case, it follows that $J(G)\cdot \O_{n+k+r}=J_{y,z}(G)\cdot \O_{n+k+r}$. Hence, proposition \ref{prop:formulaMrelG} yields
    \begin{equation*}
    M_\rel (G)=\dfrac{(F^*)^{-1} (J_{y,z}(G)\cdot \O_{n+k+r})}{J_y(G)}=\dfrac{J(G)}{J_y(G)},
\end{equation*}
and the claim follows. 
\end{proof}
Note that, since $\hat{f}$ is a finite mapping, the existence of a stable unfolding $F$ of $\hat{f}$ is granted. Furthermore, as it is commented in \cite{BobadillaNunoPenafort}, there is always a stable unfolding $F$ which admits an equation $G$ with the property that $G\in J(G)$. Such an equation is referred to as a \textit{good defining equation}. Indeed, if $F(u,x)=(u,f_u(x))$ is a stable unfolding of $\hat{f}$, then let $F'$ be the $1$-parameter stable unfolding of $F$ given by $F'(t,u,x)=(t,u,f_u(x))$. Let $G=0$ be a reduced equation of the image of $F$ and consider $G'(t,u,y)=e^t G(u,y)$. It therefore follows that $G'=0$ is a reduced equation defining the image of $F'$ with the property that $G'=\partial_t G'\in J(G')$. 
\section{Relation between $\dim_\C M(g)$ and  $\codimAe (X,f)$}
Let $F$ be a stable unfolding of $\hat{f}$, and consider an equation $G$ of the image of $F$, namely, $(Z,0)$, that satisfies $G\in J(G)$. Then, the last result of the previous section showed that $M_\rel (G) =J(G)/J_y(G)$. Let us relate this with the codimension of $(X,f)$, which can be determined through the formula
\begin{equation*}
    \codimAe (X,f)=\dim_\C \dfrac{\theta(i)}{ti(\theta_{n+1})+i^*(\Derlog Z)},
\end{equation*}
where $i:(\C^{n+1},0)\rightarrow (\C^{n+1+k+r},0)$ denotes the natural immersion (see \cite{Mond-Montaldi} for more details). 

Recall that $\Derlog Z = \{\xi \in \theta_{n+1+k+r}: \xi(G)=\lambda (G)\}$, and $\Derlog G = \{\xi \in \theta_{n+1+k+r}: \xi(G)=0\}$. Notice that $G\in J(G)$, so that $G=\sum_{s} a_s \partial_s G $, where $\partial_s G$ denotes the partial derivatives with respect to all the variables in $(y,z,u)\in \C^{n+1+k+r}$. Hence, the vector field $\epsilon = \sum_s a_s\partial_s$ satisfies that $\epsilon (G)=G$, where $\partial_s$ denotes the coordinate vector field associated with the $s$-th coordinate, where $s\in \{1, \ldots, n+1+k+r\}$. Furthermore, $\Derlog Z = \Derlog G \oplus \langle \epsilon \rangle$. We therefore have that  
\begin{equation*}
    \codimAe (X,f)=\dim_\C \dfrac{\theta(i)}{ti(\theta_{n+1})+i^*(\Derlog G)+i^*(\epsilon)}.
\end{equation*}
Notice that the evaluation mapping $\text{ev}:\theta_{n+1+k+r}\rightarrow J(G)$ given by $\xi \mapsto \xi (G)$ is a surjective mapping with kernel $\Derlog G$. Hence, it induces an isomorphism
\begin{equation*}
    \dfrac{\theta_{n+1+k+r}}{\Derlog G} \cong J(G).
\end{equation*}
Thus, 
\begin{equation*}
    \dfrac{\theta_{n+1+k+r}}{\langle \frac{\partial}{\partial y_1}, \ldots, \frac{\partial}{\partial y_{n+1}} \rangle +\Derlog G} \cong \dfrac{J(G)}{J_y (G)}=M_\rel (G).
\end{equation*}
Tensoring with $\O_{k+r}/\m_{k+r}$ yields that 
\begin{equation*}
    \dfrac{\theta(i)}{ti(\theta_{n+1}) +i^*\Derlog G} \cong M(g).
\end{equation*} 
Now, notice that the evaluation map acting on $\epsilon$ gives $\epsilon (G)=G$, and hence $i^*(\text{ev}(\epsilon))=i^*(G)=g$. Therefore, if $K(g)=(J(g)+(g))/J(g)$, then 
\begin{equation*}
    0 \rightarrow K(g) \rightarrow \dfrac{\theta(i)}{ti(\theta_{n+1}) +i^*\Derlog G} \rightarrow \dfrac{\theta(i)}{ti(\theta_{n+1}) +i^*\Derlog G +i^*(\epsilon)}\rightarrow 0
\end{equation*}
is a short exact sequence. Indeed, the evaluation map satisfies that $\text{ev}(ti(\theta_{n+1}))=J(g)$ and that $\text{ev}(i^*\Derlog G)=0$. Hence, the evaluation map yields an isomorphism 
\begin{equation*}
    \dfrac{ti(\theta_{n+1}) +i^*\Derlog G+i^*(\epsilon)}{ti(\theta_{n+1}) +i^*\Derlog G }\cong \dfrac{J(g)+(g)}{J(g)}=K(g).
\end{equation*}
After taking lengths in the exact sequence, and taking into account the previous assertions, it follows that 
\begin{theorem}\label{relation} In the above conditions, 
\begin{equation*}
    \dim_\C M(g) = \dim_\C K(g)+\codimAe (X,f). 
\end{equation*}
In particular, $\dim_\C M(g)\geq \codimAe (X,f)$, with equality in case that $g$ is weighted homogeneous.
\end{theorem}
Furthermore, this formula shows that $\dim_\C M(g)$ only depends on the isomorphism class of $g$, and neither depends on the mapping $f$ nor on the chosen extension $\hat{f}$. 

When either $(n,n+1)$ are nice dimensions or in corank one, all stable singularities are weighted homogeneous (see \cite{juanjo}). This gives the following direct consequence of Theorem \ref{relation}:

\begin{corollary}\label{nice} Assume that either $(n,n+1)$ are nice dimensions or $(X,f)$ has corank one. Then, $M(g)=0$ if and only if $(X,f)$ is stable.
\end{corollary}

\section{The Jacobian module and the Mond conjecture}\label{section:5_3}
This section is the centerpiece of this paper. We present one of the main results we aim to establish, namely, a formula for the image Milnor number expressed in terms of the Samuel multiplicity of the module $M_\rel(G)$. Additionally, we prove that the generalised Mond conjecture holds provided the module $M_\rel (G)$ is Cohen-Macaulay. 


This theorem is the key ingredient to show the main result of the article \cite{BobadillaNunoPenafort}, which is Theorem 6.1, and whose proof can be easily adapted to this new setting:

\begin{theorem} Assume that either $(n,n+1)$ are nice dimensions or $(X,f)$ has corank one. Let $F$ be a stable unfolding of $\hat{f}$. With the notations of Section \ref{section3}, $$\mu_I(X,f)=e(\m_{k+r},M_\rel(G)),$$
that is, $\mu_I(X,f)$ equals the Samuel multiplicity of the $\O_{k+r}$-module $M_\rel (G)$ with respect to $\m_{k+r}$.
\end{theorem}
\begin{proof} Take a representative of $F$ and let $w\in \C^k \times \C^r$ be a generic value. The conservation of multiplicity implies that
    \begin{equation*}
        e\Big(\m_{k+r}, M_\rel (G)\Big)=\sum_{p\in B_\epsilon} e\Big( \m_{k+r,w}, M_\rel (G)_{(p,w)} \Big).
    \end{equation*}
    In order to compute the previous multiplicity, let us take first the points $p\in Y_w\cap B_\epsilon$, where $Y_w$ is the image of $f_w:X_w\rightarrow \C^{n+1}$. Since the module $M_\rel (G)$ specialises to $M(g)$, it follows that 
    \begin{equation*}
        M_\rel (G)_{(p,w)}\otimes \dfrac{\O_{k+r,w}}{\m_{k+r,w}}\cong M(g_w)_p
    \end{equation*}
    as $\C$-vector spaces in virtue of Remark \ref{remark:specialisation2}. If $F$ is a stable unfolding, it follows that, provided $w$ is generic, $f_w$ is a stable mapping, and hence $ M(g_w)_p=0$ for each $p\in Y_w$, by Corollary \ref{nice}. Hence, the points $p\in Y_w \cap B_\epsilon$ do not contribute to the term $e(\m_{k+r}, M_\rel (G))$. 
    
    On the other hand, if $p\in B_\epsilon / Y_w$, then 
    \begin{equation*}
        M_\rel (G)_{(p,w)}  = \dfrac{\O_{\C^{k+r}\times B_\epsilon, (p,w)}}{J_y (G)}
    \end{equation*}
    is a module with dimension $\geq k+r$. Indeed, this follows from the exact sequence that defines $M_\rel (G)$, since the localisation of $C(F)/(J_{y,z}(G)\cdot \O_{n+k+r})$ is zero if $p\notin Y_w$, and $\F_1(F)$ localises to $\O_{\C^{k+r}\times B_\epsilon, (p,w)}$. Moreover, 
    \begin{equation*}
        M_\rel (G)_{(p,w)}\otimes \dfrac{\O_{k+r,w}}{\m_{k+r,w}} \cong  \dfrac{\O_{B_\epsilon, p}}{J(g_w)}
    \end{equation*}
    is a module with dimension $0$, since it has finite length due to the fact that $g_u$ has isolated singularity. This implies that the dimension of $M_\rel (G)$ is $\leq k+r$. Hence, $M_\rel (G)_{(p,u)}$ is a complete intersection ring, and in particular a Cohen-Macaulay $\O_{k+r}$-module of dimension $k+r$. Hence, 
    \begin{align*}
        e\Big( \m_{k+r,w}, M_\rel (G)_{(p,w)} \Big) &= \dim_\C \left( M_\rel (G)_{(p,w)} \otimes \dfrac{\O_{k+r,w}}{\m_{k+r,w}} \right) \\&= \dim_\C \dfrac{\O_{B_\epsilon, p}}{J(g_w)} = \mu(g_w, p). 
    \end{align*}
    By Siersma's Theorem \ref{Siersma}, it follows that $\sum_{p\in B_\epsilon / Y_w} \mu(g_w,p) = \mu_I(X,f)$. 
\end{proof}
Once this result has been proven, the following theorem is now an immediate consequence:
\begin{theorem}\label{theorem:MC} In the above conditions, the following statements are equivalent and imply the generalised Mond conjecture for $(X,f)$:
\begin{enumerate}
    \item $\dim_\C M(g)=\mu_I(X,f)$,
    \item $M_\rel(G)$ is a Cohen-Macaulay $\O_{k+r}$-module of dimension $k+r$.  
\end{enumerate}
Furthermore, if $g$ is weighted homogeneous and satisfies the generalised Mond conjecture, then the above assertions hold. 
\end{theorem}
    \begin{proof} Recall that the Samuel multiplicity of an $R$-module $M$ is generally smaller than the length of $M/(\m\cdot M)$, with equality if and only if $M$ is a Cohen-Macaulay module with the same dimension as $R$. Hence, $M_\rel (G)$ is Cohen-Macaulay of dimension $k+r$ if and only if
    \begin{align*}
        \mu_I(X,f)&=e\Big(\m_{k+r}, M_\rel (G)\Big) = \dim_\C \dfrac{M_\rel (G)}{\m_{k+r} M_\rel (G)}\\&= \dim_\C M_\rel (G)\otimes \dfrac{\O_{k+r}}{\m_{k+r}}=\dim_\C M(g).
    \end{align*}
    Therefore, if one of the items hold, it then follows that 
    \[
    \mu_I(X,f)=\dim_{\C} M(g) = \dim_\C K(g) + \codimAe (X,f) \geq \codimAe (X,f),
    \] 
    and hence, the generalised Mond conjecture holds for $(X,f)$. Moreover, if $g$ is weighted homogeneous, then $K(g)=0$. Thus, if the generalised Mond conjecture holds for $(X,f)$, then 
    \[\mu_I(X,f)=\codimAe(X,f)=\dim_\C M(g)-\dim_\C K(g)=\dim_\C M(g).
    \] 
    Hence, the above assertions hold. 
    \end{proof}
\begin{remark} Recent work by Nuño-Ballesteros and Giménez Conejero \cite{roberto2} has shown that the requirement for $M_\rel(G)$ to be $k+r$-dimensional can be eliminated from the second condition. Thus, we have that $\dim_\C M(g)=\mu_I(X,f)$ if and only if $M_\rel (G)$ is a Cohen-Macaulay module. Indeed, if the dimension of $M_\rel (G)$ is strictly less than $k+r$, then $\mu_I(X,f)$ is zero, which, according to \cite{roberto2}, implies that $(X,f)$ is a stable map-germ, and hence $\codimAe(X,f)=0$.
\end{remark}
This result shows that in order to establish the validity of the Mond conjecture for $f$, it suffices to check the Cohen-Macaulay property of the module $M_\rel(G)$. While it has been observed that this module is Cohen-Macaulay in every computed example, it remains an open question to provide a rigorous proof of this statement.

\section{Proof of the generalised Mond conjecture for \textit{n\,=\,2}}\label{section:5_4}
In this section, we achieve our main objective of the article, which is to show the validity of the generalised Mond conjecture for $n=2$ by employing the results from the previous section.

Building on the main result of the previous section, to establish the generalised Mond conjecture for mappings $f:(X,0)\rightarrow (\C^3,0)$ where $(X,0)$ is a 2-dimensional \textsc{icis,} it suffices to check that $M_\rel (G)$ is a Cohen-Macaulay module of dimension $k+r$. Notice first that the dimension of $M_\rel (G)$ is $\leq k+r$, due to the fact that 
\begin{equation*}
    \dim_\C M_\rel (G)\otimes \dfrac{\O_{n+1+k+r}}{\m_{k+r}\cdot \O_{n+1+k+r}} =\dim_\C M(g) <+\infty.
\end{equation*}
Therefore, it is enough to show that $\depth M_\rel (G) \geq k+r$. To achieve this, we apply the depth lemma to the exact sequence
\begin{equation*}
    0\rightarrow M_\rel(G) \rightarrow \dfrac{\F_1(F)}{J_y(G)}\rightarrow \dfrac{C(F)}{J_{y,z}(G)\cdot \O_{n+k+r}} \rightarrow 0
\end{equation*}
to obtain that
\begin{equation*}
    \depth M_\rel (G)\geq \min \left\{ \depth \dfrac{\F_1(F)}{J_y (G)}, \depth \dfrac{C(F)}{J_{y,z} (G)\cdot \O_{n+k+r}} + 1 \right\}.
\end{equation*}
Hence, both terms of the minimum have to be shown to be greater than or equal to $k+r$. For the module $C(F)/J_{y,z} (G)\cdot \O_{n+k+r}$, it has been checked in Remark 3.9 of \cite{BobadillaNunoPenafort} that this module is Cohen-Macaulay of dimension $n+k+r-2$ provided $n\geq 2$, due to the fact that it is isomorphic to the determinantal ring $\O_{n+k+r}/R(F)$, where $R(F)$ denotes the ramification ideal of $F$. In particular, it follows that 
\begin{equation*}
    \depth \dfrac{C(F)}{J_y (G)\cdot \O_{n+k+r}} + 1 = n+k+r-1 \geq k+r.
\end{equation*}
Therefore, it is enough to verify that $\depth \F_1(F)/J_y (G) \geq k+r$. Notice that, up to this point, the assumption $n=2$ has not been used yet. In general, the module $\F_1(F)/J_y (G)$ has dimension 
\begin{equation*}
    \dim \dfrac{\F_1(F)}{J_y (G)} = \max \left\{\dim M_\rel (G), \dim \dfrac{C(F)}{J_y (G)\cdot \O_{n+k+r,0}} \right\} =n+k+r-2,
\end{equation*}
due to the fact that $\dim M_\rel (G) \leq k+r \leq n+k+r-2$ provided $n\geq 2$.
Therefore, it is not expected that $M_\rel(G)$ will be Cohen-Macaulay for every $n\geq 2$. The only case in which we can expect this is when $n=2$, since, in this case, its dimension is precisely $k+r$. To verify this claim, we make use of Pellikaan's Theorem:
\begin{theorem}[Pellikaan, Section 3 of \cite{Pellikaan}] If $J\subset F\subset R$ are ideals of $R$ where $J$ is generated by $m$ elements, $\grade (F/J)\geq m$ and $\pd\, (R/F)=2$, then $F/J$ is a perfect module and grade$\,(F/J)=m$.
\end{theorem}
This result plays a crucial role to show that the module is indeed Cohen-Macaulay, as it is proven in the following proposition:
\begin{theorem}\label{theorem:n=2} If $n=2$, then $\F_1(F)/J_y (G)$ is a Cohen-Macaulay module. 
\end{theorem}
\begin{proof} We follow the notation of the previous result, taking $R=\O_{n+1+k+r}$, $F=\F_1(F)$ and $J=J_y (G)$. It follows from the proof of Theorem 3.4 of \cite{Mond-Pellikaan} that $R/F=\O_{n+1+k+r}/\F_1(F)$ is a determinantal ring of dimension $n+k+r-1$, and hence Cohen-Macaulay. Hence, by the Auslander-Buchsbaum formula,
\begin{align*}
    \pd\, (R/F) &= \depth R - \depth R/F = \dim R - \dim R/F \\&= n+1+k+r - (n+k+r-1) = 2. 
\end{align*}
Moreover, $J=J_y (G)$ is clearly generated by $n+1=3$ elements, namely the partial derivatives of $G$ with respect to the variables $y_1, \ldots, y_{n+1}$. Furthermore, 
\begin{align*}
    \grade (F/J) &= \depth (\Ann (F/J), \O_{n+1+r})= \height \Ann (F/J) = \\  
    &= \dim \O_{n+1+r}-\dim F/J= n+1+r-(n+r-2)=3.
\end{align*}
Thus, Pellikaan's Theorem states that $F/J$ is a perfect $\O_{n+1+k+r}$-module. Furthermore, since $\O_{n+1+k+r}$ is a local Cohen-Macaulay ring, then $F/J$ is Cohen-Macaulay. 
\end{proof}

Now the main result of this article follows easily as an application of the results presented in the previous section.

\begin{proof}[Proof of Theorem \ref{thrm:main}] Theorem \ref{theorem:n=2} implies that $M_\rel (G)$ is a Cohen-Macaulay module. Hence, Theorem \ref{theorem:MC} implies that the generalised Mond conjecture holds for $(X,f)$. 
\end{proof}
In this setting, the image Milnor number therefore satisfies that $\mu_I(X,f)=\dim_\C M(g)$. This gives an operative method to compute this number, as the following example shows:
\begin{example} Let $(X,0)\subset (\C^3,0)$ be the hypersurface defined by $x^3+y^3-z^2=0$ and let $f:(X,0)\rightarrow (\C^3,0)$ be the $\A$-finite mapping $f(x,y,z)=(x,y,z^3+xz+y^2)$. In this case, $\hat{f}(x,y,z)=(x,y,z^3+xz+y^2, x^3+y^3-z^2)$ turns out to be a stable mapping. With \textsc{Singular} \cite{DGPS}, one easily obtains that $\codimAe(X,f)=6$ and $\mu_I(X,f)=\dim_\C M(g)=6$. 
\end{example}


\bibliographystyle{plain}
\bibliography{referencearticle2}

\end{document}